\documentclass[reqno,a4paper,12pt]{amsart}
\usepackage{latexsym, amsmath, amssymb, amsthm, a4}

\usepackage[active]{srcltx}

\usepackage{amsfonts}

\usepackage{mathrsfs}

\newtheorem{theorem}{Theorem}[section]

\newtheorem{definition}{Definition}

\newtheorem{example}{Example}[section]
\newtheorem{remark}[theorem]{Remark}

\renewcommand{\a}{\alpha}
\renewcommand{\b}{\beta}

\newcommand{\z}{\zeta}

\renewcommand{\k}{\kappa}

\renewcommand{\l}{\lambda}

\newcommand{\m}{\mu}

\newcommand{\s}{\sigma}

\newcommand{\F}{\Phi}

\newcommand{\Z}{{\mathbb Z}}

\newcommand{\DD}{{\mathbb{D}}}

\newcommand{\Db}{{\mathbf D}}

\newcommand{\Fb}{{\mathbf F}}

\newcommand{\Ib}{{\mathbf I}}

\newcommand{\Tb}{{\mathbf T}}

\newcommand{\Sf}{\mathfrak S}

\newcommand{\Ec}{{\mathcal E}}

\newcommand{\dist}{{\rm dist}\,}

\newcommand{\supp}{\hbox{{\rm supp}}\,}

\newcommand{\pd}{\partial} 
\newcommand{\pa}{\bar{\partial}}

\newcommand{\Tr}{\operatorname{Tr\,}}
\newcommand{\tr}{\operatorname{tr\,}}

\theoremstyle{plain}

\newtheorem*{theorem*}{Theorem}

\theoremstyle{definition}

\newcommand{\la}{\langle}
\newcommand{\ra}{\rangle}

\newcommand{\beq}{\begin{equation}}
\newcommand{\eeq}{\end{equation}}

\numberwithin{equation}{section}
\numberwithin{figure}{section}

\begin{document}

\title[Trace Class singular]{Trace class Toeplitz operators with singular symbols}

\author{Grigori Rozenblum }

\address{ Chalmers University of Technology and The University of Gothenburg (Sweden); St.Petersburg State University, Dept. Math. Physics (St.Petersburg, Russia)}

\email{grigori@chalmers.se}
\author {Nikolai Vasilevski}
\address{Cinvestav (Mexico, Mexico-city)}
\email{nvasilev@cinvestav.mx}

\dedicatory{To Armen Sergeev in occasion of his $70^{th}$ anniversary}
\begin{abstract}
We characterize the trace class membership of Toeplitz operators with distributional symbols acting on the Bergman space on the unit disk. The Berezin transform of distributions, introduced in the paper, yields a formula for the trace. Several instructive examples are also given.
\end{abstract}
\thanks{The first-named author is supported by grant RFBR No 17-01-00668. He is grateful to CINVESTAV for hospitality and support.}

\maketitle



\section{Introduction}

The classical Bergman space $\mathcal{A}^2(\mathbb{D})$ is a closed subspace in $L_2(\mathbb{D}, dA)$, where $dA(z)=\frac{1}{\pi}dxdy$ is the normalized Lebesgue measure, which consists of functions analytic on the unit disk $\mathbb{D}$. It is a reproducing kernel Hilbert space, whose reproducing kernel is given by $K_z(w) = (1-\overline{z}w)^{-2}$. The orthogonal Bergman projection $P: L_2(\mathbb{D}, dA) \rightarrow \mathcal{A}^2(\mathbb{D})$ is given then by $(Pf)(z) = \langle f, K_z \rangle$.

Classical Toeplitz operator $T_a$, with symbol $a \in L_{\infty}(\mathbb{D})$, is defined as the compression of the multiplication by $a$ operator onto the Bergman space, i.e., it acts on $\mathcal{A}^2(\mathbb{D})$ as follows $T_af= P(af)$. This classical definition of Toeplitz operators has been extended then to various more general situations, in particular, to weighted Bergman spaces, to unit ball case and to more general multidimensional domains of holomorphy, to measure and distributional symbols, etc. 

Toeplitz operators with strongly singular, distributional symbols were considered first, probably,  in \cite{AlexRoz} and \cite{PeTaVi3}. Recently, in \cite{RV1,RV2}, a new approach to Toeplitz operators has been proposed. The authors use the language of sesquilinear forms, which permit them not only to cover all previous generalizations in a unified form, but also extend the notion of Toeplitz operators to highly singular symbols.

The approach of \cite{RV1,RV2} goes as follows (for proofs and details see \cite{RV1,RV2}). Let $F(\cdot,\cdot)$ be a bounded sesquilinear form on a reproducing kernel Hilbert space $\mathcal{A}$, with $K_z$ being its reproducing kernel. Then Toeplitz operator $T_F$, defined by the form $F$, acts on $\mathcal{A}$ as follows: $(T_ff)(z) = F(f,K_z)$. In particular, \cite{RV2} treats Toeplitz operators defined by $k$-Carleson measures for derivatives. Recall in this connection that the measure $\mu$ is called $k$-Carleson if
\begin{equation*}
 \left|\int_{\mathbb{D}} |f^{(k)}(z)|^2 d\mu(z) \right| \leq C \|f\|^2_{\mathcal{A}^2(\mathbb{D})},
\end{equation*}
where the constant $C$ does not depend on $f \in \mathcal{A}^2(\mathbb{D})$.

In the present paper we start the study the Schatten class membership of Toeplitz operators defined by sesquilinear forms. To demonstrate the ideas and to make our considerations more transparent avoiding unnecessary technicalities, we consider in this paper the simplest case: non-weighted Berman space $\mathcal{A}^2(\mathbb{D})$, Toeplitz operators defined by $k$-Carleson measures, and the trace class membership.
More general cases will be considered elsewhere.

Recall, for completeness, that the trace and Schatten class membership of Toeplitz operators with more regular symbols has been studied, first in \cite{Luecking}, and then in  \cite{ElFall,Pelaez,Zhu}, etc.

\section{A sufficient trace class condition}

Following the pattern in \cite{RV1,RV2}, we introduce now sesquilinear forms involving derivatives of functions $f,g$ and thus corresponding to derivatives of Carleson measures.
\begin{definition} Let $\m$ be a regular complex measure on $\mathbb{D}$ and $\a,\b$ be two nonnegative integers. We denote by $\Fb_{\a,\b,\m}$ the sesquilinear form
\begin{equation}\label{FormDer}
   \Fb_{\a,\b,\m}[f,g]=(-1)^{\a+\b}\int_{\DD}\partial^{\a}f(z)\overline{\partial^\b g(z)}d\m(z), \, f,g\in\mathcal{A}^2(\mathbb{D}).
\end{equation}
\end{definition}

In \cite{RV2} a condition was found for the sesquilinear form \eqref{FormDer}  to be bounded or compact. This condition is sufficient for any measure $\m$, and it is also necessary for a positive measure and $\a=\b$. Such condition is expressed using the notion of $k$--Carleson measures. Our aim is to find a condition for a $k$-Carleson measure to generate a trace class operator. We denote by $\Tb_{\a,\b,\m}:=\Tb_{\Fb_{\a,\b,\m}}$ the Toeplitz operator defined by this form, or, shorter, $\Tb_{\Fb}$, with $\Fb=\pd^{\a}\pa^{\b}\m$.

Note first that for a measure $\mu$ with compact support in $\DD$ the Toeplitz operator $\Tb_{\a,\b,\m}$ has singular numbers decaying exponentially, therefore $\Tb_{\a,\b,\m}$ belongs to all Schatten classes $\Sf^p$, $p>0$, for any $\a,\b$. We recall the reasoning in \cite{RV2}.

\begin{theorem}\label{Compact Support} Let $\m$ be a measure with compact support in $\DD$ and let  $\a,\b$ be some nonnegative integers.
Then
 For any $f,g\in\mathcal{A}^2(\mathbb{D})$, the integral in \eqref{FormDer} converges, moreover,
\begin{equation}\label{FormDerTransf}
   \Fb[f,g]= \Fb_{\a,\b,\m}[f,g]=(\partial^{\a}\bar{\partial}^{\b}\m, f\bar{g}),
\end{equation}
where the derivatives are understood in the sense of distributions in $\Ec'(\DD)$ and the parentheses mean the intrinsic paring of $\Ec'(\DD) $ and $\Ec(\DD)$. The singular numbers $s_n(\Tb_{\Fb})$ of the operator $\Tb_{\Fb}$ satisfy the estimate
\begin{equation}\label{snumbers}
    s_n(\Tb_{\Fb})\le C \exp(-n\s),
\end{equation}
where $\s>0$ is some constant determined by the measure $\m$ and integers $\a,\b$.

\end{theorem}
\begin{proof}  Due to the Ky Fan's inequality for singular numbers of compact operators, it is sufficient to establish \eqref{snumbers} for a positive measures $\m$.
Consider a closed disk $\Db\subset \DD$ with radius $R$ such that the support of $\m$ lies strictly inside $\Db$, thus $\dist(z,\pd \Db)>r>0$ for all $z\in\supp \m$.
The Cauchy integral formula implies that for any $z\in\supp \m$ and any $\a\in\Z_+$,
\begin{equation}\label{est in disk}
|\partial^\a f(z)|^2\le C_{\a}\int_{\partial \Db}|f(\z)|^2dA(\z),
\end{equation}
for any function $f\in\mathcal{A}^2(\mathbb{D})$. Here the constant $C_\a$ depends only on  $\a$, but does not depend on $f$ and $z$. By the same reason,  for any $z\in\supp\m$, the estimate \eqref{est in disk} holds, with the domain of integration replaced by $\Db'$, the disk concentric with $\Db$, but with radius $R-r/2$. By the Cauchy-Schwartz inequality,
\begin{equation}\label{est in disk3}
    |\Fb_{\a,\b,\m}[f,g]|\le \left(\int_{\textrm{supp}\, \m}|\partial^\a f(z)|^2 d|\m(z)|\right)^{\frac12}\left(\int_{\textrm{supp}\, \m}|\partial^\b g(z)|^2 d|\m(z)|\right)^{\frac12},
\end{equation}
for all $f,g\in\mathcal{A}^2(\mathbb{D})$, and then, due to \eqref{est in disk}, we obtain the estimate
\begin{equation}\label{est 4}
  |\Fb_{\a,\b,\m}[f,g]|\le C'_{\a}C'_{\b}|\m|(\Db)\|f\|_{L^2(\Db')}\|g\|_{L^2(\Db')}.
\end{equation}
The last relation means that the sesquilinear form $\Fb_{\a,\b,\m}$ is bounded not only in $\mathcal{A}^2(\mathbb{D})$, but in the  spaces $\mathcal{A}^2(\Db)$ and $\mathcal{A}^2(\Db')$ as well.

Now we represent our Toeplitz operator $\Tb_\Fb$ as the composition
\begin{equation}\label{est5}
    \Tb_{\Fb}=\Tb_{\Fb}(\mathcal{A}^2(\Db'))\Ib_{\Db'\Rrightarrow \Db}\Ib_{\Db\Rrightarrow\DD},
\end{equation}
where $\Ib_{\Db'\Rrightarrow \Db}:\mathcal{A}^2(\Db)\to\mathcal{A}^2(\Db')$, $\Ib_{\Db\Rrightarrow \DD}:\mathcal{A}^2(\DD)\to\mathcal{A}^2(\Db)$ are  operators generated by the restriction of functions on a larger set to the corresponding smaller set. The equality \eqref{est5} can be easily checked by writing the sesquilinear forms of operators on the left-hand and on the right-hand side. Finally, of the three operators on the right-hand side in \eqref{est5}, the first and the third ones are bounded, while the middle one, the operator generated by the embedding of the disk $\Db'$ to $\Db$, is known to have an exponentially fast decaying sequence of singular numbers see, e.g., \cite{parfenov}.
\end{proof}

The above result demonstrates that the 'quality' of the Toeplitz operator with distributional symbol, in the sense of its belonging to Schatten classes of compact operators, is determined by the behavior of the symbol near the boundary.

\section{The Berezin transform and the trace class}
Recall that, given a bounded operator $\Tb$ on the Bergman space $\mathcal{A}^2(\DD)$, its Berezin transform is defined by $\widetilde\Tb(z)=\la \Tb \k_z,\k_z\ra$,
where $\k_z(w)=\frac{1-|z|^2}{(1-\overline{z}w)^2}$ is the normalized reproducing kernel for $\mathcal{A}^2(\mathbb{D})$. The M\"obius invariant measure on $\DD$ is denoted by $d\l(z)=(1-|z|^2)^{-2} dA(z)$, where, as previous, $dA(z)= \frac{1}{\pi}dxdy$.

We recall now the basic result by K.Zhu (see \cite{Zhu}, Theorem 6.4 and Corollary~6.5).
\begin{theorem}\label{Zhu.gener.th}
Let $\Tb$ be a non-negative bounded operator in $\mathcal{A}^2(\mathbb{D})$. Then in the equality
\begin{equation}\label{fundamental property}
    \Tr \Tb=\int_{\DD}\widetilde\Tb(z)d\l(z),
\end{equation}
one side is finite as long as the other one is finite. The same assertion holds true if  the nonnegativity condition is replaced by the requirement that $\Tb$ belongs to the trace class $\Sf^1$.
\end{theorem}
We will apply Theorem \ref{Zhu.gener.th} to Toeplitz operators generated by measures. We start with the symmetric positive case.
\begin{theorem}\label{Th.distr} Let $\m$ be a regular non-negative $k$-Carleson measure, $k\ge 1$. Consider the Toeplitz operator $\Tb\equiv\Tb_{k,k;\m}$ generated by the distribution $\F=\pd^k\pa^k\m$. Then the operator $\Tb$ belongs to the trace class as soon as
\begin{equation}\label{finite}
    \int_{\DD}(1-|w|^2)^{-2-2k}d\m(w)<\infty,
\end{equation}
and in this case
\begin{equation}\label{trace}
    \tr \Tb=\left(\F,\frac{1}{(1-|w|^2)^2}\right).
\end{equation}
\end{theorem}
\begin{proof}We start by establishing the trace class property. Our operator is nonnegative, since the sesquilinear form $\Fb_{k,k;\mu}(f,f)$ is nonnegative. By \eqref{fundamental property},
it is sufficient to prove that
\begin{equation}\label{fin1}
    \int_{\DD}\widetilde\Tb(z)d\l(z)<\infty.
\end{equation}
We have
\begin{gather}\label{fin2}
    \widetilde\Tb(z)=\la\Tb \k_z(.),\k_z(.)\ra=\int_{\DD}(\m, \pd^k\pa^k\k_z\overline{\k_z})\\\nonumber =((k+1)!)^2|z|^{2k}(1-|z|^2)^{2}\int_{\DD}|1-\bar{w}z|^{-4-2k}d\mu.
\end{gather}
We substitute the last expression into \eqref{fin1}, to obtain the finiteness condition
\begin{equation*}
    \int_{\DD} |z|^{2k} \int_{\DD}|1-\bar{w}z|^{-4-2k}d\m(w) dA(z)<\infty;
\end{equation*}
the latter, by the Fubini theorem, is equivalent to
\begin{equation}\label{fin3}
    \int_{\DD}d\mu(w)\int_{\DD}|z|^{2k}|1-\bar{w}z|^{-4-2k}dA(z)<\infty.
\end{equation}
The inner integral, the one in $z$ variable, can be majorated by $C(1-|w|^2)^{-2-2k}$. Therefore, the expression in \eqref{fin3} is no greater than $\int_{\DD}(1-|w|^2)^{-2-2k}d\m(w)$, just what we needed.

Now we can justify the formula \eqref{trace} for the trace of the operator $\Tb$.

By \eqref{fundamental property}, \eqref{fin1}, we have
\begin{equation}\label{trace5}
    \tr \Tb=\int_{\DD}(\F, |\k_z(.)|^2)(1-|z|^2)^2dA(z)=\int_{\DD}\int_{\DD}d\m(w)\pd^k\pa^k|1-z\bar{w}|^{-4} dA(z)
\end{equation}
As it was just shown, under the condition \eqref{finite} the double integral converges absolutely and we may change the order of integration in \eqref{trace5}, obtaining
\begin{equation}\label{trace6}
     \tr \Tb =\left(\F, \int_{\DD}{|1-z\bar{w}|^4}{dA(z)}\right)=\left(\F, (1-|w|^2)^{-2} \int_{\DD}\frac{(1-|w|^2)^2}{|1-z\bar{w}|^4}dA(z)\right).
\end{equation}
The inner integral in \eqref{trace6} is exactly the Berezin transform of the constant function $\pmb{\pmb{1}}$. Since this function is harmonic, its Berezin transform coincides with itself. This gives us
\begin{equation}\label{trace7}
    \tr \Tb=\left(\F, \frac1{(1-|w|^2)^2} \pmb{\pmb{1}}\right)=\left(\F, \frac1{(1-|w|^2)^2}\right).
\end{equation}
\end{proof}

Now we are able to resolve the trace class problem for more general distributions~$\F$.

\begin{theorem}\label{FinitTh.2}Let the distributional symbol $\F$ have the form $\F=\pd^\a\pa^\b \m$, where $\m$ is a $k$-Carleson measure, with $2k=\a+\b$. Then, as soon as
\begin{equation}\label{general}
    \int_{\DD}(1-|w|^2)^{-2k-2}d|\m|(w)<\infty,
\end{equation}
the operator $\Tb_{\F}$ belongs to the trace class $\Sf^1$ and
\begin{equation}\label{trace88}
    \tr \Tb=\left(\F,\frac{1}{(1-|w|^2)^2}\right).
\end{equation}
\end{theorem}
\begin{proof}As usual, it suffices to consider the case of a nonnegative measure $\m$.
We represent the sesquilinear form of the operator $\Tb_{\F}$ as
\begin{equation}\label{form1}
    \Tb_\F(f,g)=(-1)^{\a+\b}\int_{\DD}((1-|w|^2)^{\a-k}\pd^{\a}f(w))\times\overline{((1-|w|^2)^{\b-k}\pd^{\a}g(w))}d\m(w).
    \end{equation}
We introduce two  operators acting from the Bergman space $\mathcal{A}^2(\mathbb{D})$ to the weighted $L^2$ space, $L^2_{\m}(\DD)$, in the following way:
\begin{equation}\label{operators}
    H_1: f(w)\mapsto (1-|w|^2)^{\a-k}\pd^{\a}f(w);\, H_2: g(w)\mapsto (1-|w|^2)^{\b-k}\pd^{\b}g(w).
\end{equation}
Using these operators, the sesquilinear form  $\Tb_\F(f,g)$ can be written as
\begin{equation}\label{form4}
    \Tb_\F(f,g)=\la H_1 f, H_2g\ra_{L^2_{\m}(\DD)}=\la H_2^*H_1 f, g\ra_{\mathcal{A}^2(\mathbb{D})}.
\end{equation}
This means that the Toeplitz operator $\Tb_\F$ is factorized as
\begin{equation}\label{form5}
    \Tb_F= H_2^*H_1.
\end{equation}
Therefore, we need to prove that both operators $H_1,H_2$ belong to the Hilbert-Schmidt class $\Sf^2$. Consider, for example, $H_1$. We will show now that the nonnegative operator $H_1^*H_1$ (acting in the Bergman space $\mathcal{A}^2(\mathbb{D})$) belongs to the trace class $\Sf^1$. In fact, the sesquilinear  form of this operator equals
\begin{equation}\label{Form6}
    \la H_1^* H_1 f,g \ra_{\mathcal{A}^2(\mathbb{D})}=\la H_1 f, H_1 g\ra_{L^2_\m}=\int_{\DD} (\pd^\a f\pd^{\a}g)(1-|w|^2)^{2\a-2k}d\mu.
\end{equation}
We have already established  in Theorem \ref{Th.distr} (applied to the measure $(1-|w|^2)^{2\a-2k} d\mu(w)$ instead of $d\m(w)$) that the condition \eqref{general} suffices for the operator defined by the sesquilinear form \eqref{Form6} to belong to $\Sf^1$. Therefore, the operator $H_1$ is a Hilbert-Schmidt one. The same reasoning takes care of $H_2$.  Finally this means that $H_2^*H_1$ is a trace class operator.

To prove \eqref{trace88}, we need, similar to the reasoning in Theorem \ref{Th.distr}, to show that it is possible to change the order of integration in the expression

\begin{equation}\label{trace9}
    \tr \Tb=\int_{\DD}(\F, |\k_z(.)|^2)(1-|z|^2)^2dA(z)=\int_{\DD}\int_{\DD}d\m(w)\pd^\a\pa^\b|1-z\bar{w}|^{-4} dA(z).
\end{equation}
We represent \eqref{trace9} as
\begin{equation}\label{trace10}
    \tr \Tb = (\a+1)!(\b+1)!\int_{\DD}z^{\b}\bar{z}^{\a}\int_{\DD}(1-z\bar{w})^{-2-\b}(1-\bar{z}{w})^{-2-\a}d\m(w) dA(z),
\end{equation}
and then transform, similarly to \eqref{form1}:

\begin{equation}\label{trace11}
\tr \Tb = (\a+1)!(\b+1)!\int_{\DD}z^{\b}\bar{z}^{\a}\int_{\DD}(1-|w|^2)^{\b-k}(1-z\bar{w})^{-2-\b}(1-|w|^2)^{\a-k}(1-\bar{z}{w})^{\a}d\m(w) dA(z).
\end{equation}
The last integral is estimated by applying the Cauchy-Schwartz inequality,
\begin{gather}\label{trace12}
   | \tr \Tb |\le (\a+1)!(\b+1)!\left(\int_{\DD}(1-|w|^2)^{2\b-2k}|1-z\bar{w}|^{-4-2\b}d\m(w) dA(z)\right)^{\frac12}\times \\\nonumber\left(\int_{\DD}(1-|w|^2)^{2\a-2k}|1-z\bar{w}|^{-4-2\a}d\m(w) dA(z)\right)^{\frac12},
\end{gather}
and both integrals in \eqref{trace12} converge under our conditions. Therefore, the double integral  in \eqref{trace11} converges absolutely and we may apply the Fubini theorem to change the order of integration. After this, the proof of \eqref{trace88} goes in the same way as in Theorem \ref{Th.distr}.
\end{proof}
\begin{remark}\label{rem1}We might have tried to  prove Theorem \ref{FinitTh.2} in a more straightforward way, actually, mimicking the reasoning in Theorem \ref{Th.distr}, i.e., by estimating the Berezin transform of the operator $\Tb_\F$ with $\F=\pd^\a\pa^\b \m$. This approach would lead us to estimating the integral
\begin{equation}\label{Form7}
    \int_{\DD}\int_{\DD}\bar{z}^\a z^\b(1-\bar{z}w)^{-2-\a}(1-z\bar{w})^{-2-\b}d\m(w) d A(z).
\end{equation}
However, on the one hand, we do not know a-priori that the integral in \eqref{Form6} converges absolutely, therefore, applying the Fubini theorem is not justified. Moreover, without the nonnegativity assumption, the formula \eqref{fundamental property} is not justified either -- one must know first that the operator is a trace class one.
However, having Theorem \ref{FinitTh.2} proved, we know already that the operator is trace class and therefore we can legally use formula \eqref{fundamental property}. After this, the change in the integration order becomes legal.
 \end{remark}

\section{Examples}

We illustrate the above on several examples involving $k$-Carleson measures considered in \cite{RV2}.

In what follows we will constantly use that, for $\alpha,\, \beta \in \mathbb{Z}_+$,
\begin{equation*}
 \partial^{\alpha}_w \frac{1}{(1-\overline{z}w)^2} = \frac{(\alpha+1)! \,\overline{z}^{\alpha}}{(1-\overline{z}w)^{2+\alpha}} \quad \textrm{and} \quad
 \partial^{\beta}_{\overline{w}} \frac{1}{(1-z\overline{w})^2} = \frac{(\beta+1)! \,z^{\beta}}{(1-z\overline{w})^{2+\beta}}.
\end{equation*}

\begin{example} {\rm
We start with the Carleson measure $(1+|z|)^{2k}dA(z)$, then, by \cite[Proposition 6.2]{RV2}, the measure $d\mu = (1-|z|)^{2k}(1+|z|)^{2k}dA(z) = (1-|z|^2)^{2k}dA(z)$ is a $k$-Carleson measure. Now with $\alpha + \beta \leq 2k$, consider
\begin{equation*}
 \Fb_{\alpha,\beta,\mu}[f,g] = (\partial^{\alpha}\overline{\partial}^{\beta}\mu,f\overline{g}),
\end{equation*}
the corresponding Toeplitz operator $\Tb_{\alpha,\beta,\mu} := \Tb_{\Fb_{\alpha,\beta,\mu}}$, and its Berezin transform
\begin{eqnarray*}
 \widetilde{\Tb}_{\Fb_{\alpha,\beta,\mu}}(z) &=& \Fb_{\alpha,\beta,\mu}(\k_z,\k_z) = \int_{\mathbb{D}} \partial^{\alpha}_w \k_z(w) \partial^{\beta}_{\overline{w}} \overline{\k}_z(w) (1-|w|^2)^{2k}dA(w) \\
&=& \int_{\mathbb{D}} \partial^{\alpha}_w \frac{1-|z|^2}{(1-\overline{z}w)^2} \partial^{\beta}_{\overline{w}} \frac{1-|z|^2}{(1-z\overline{w})^2} (1-|w|^2)^{2k} dA(w) \\
&=& (1-|z|^2)^2 (\alpha+1)!(\beta+1)!z^{\beta}\overline{z}^\alpha \int_{\mathbb{D}} \frac{(1-|w|^2)^{2k}}{(1-\overline{z}w)^{2+\alpha}(1-z\overline{w})^{2+\beta}} dA(w).
\end{eqnarray*}
For $\alpha = \beta$, we have
\begin{eqnarray*}
 \widetilde{\Tb}_{\Fb_{\alpha,\alpha,\mu}}(z) &=& (1-|z|^2)^2 [(\alpha+1)!]^2 |z|^{2\alpha} \int_{\mathbb{D}} \frac{(1-|w|^2)^{2k}}{|1-\overline{z}w|^{2(2+\alpha)}} dA(w) \\
&=& \frac{\alpha!(\alpha+1)!|z|^{2\alpha}}{(1-|z|^2)^{\alpha}}\, B_{\alpha}\left((1-|w|^2)^{2k-\alpha}\right) (z),
\end{eqnarray*}
where $B_{\alpha}$ is the Berezin transform on weighted Bergman space $\mathcal{A}^2_{\alpha}(\mathbb{D})$.

Formula \eqref{trace} permits us to calculate easily the trace of the operator $\Tb_{\alpha,\beta,\mu}$ with $\alpha + \beta \leq 2(k-1)<2k$:
\begin{equation*}
 \mathrm{tr}\, \Tb_{\alpha,\beta,\mu} = \left( \Fb_{\alpha,\beta,\mu}, \frac{1}{(1-|w|^2)^2} \right) = \int_{\mathbb{D}} \left[\partial^{\alpha}_w\partial^{\beta}_{\overline{w}} \frac{1}{(1-w\overline{w})^2}\right](1-|w|^2)^{2k} dA(w).
\end{equation*}
In particular, for an integer or half-integer $k \geq2$ and $\alpha=\beta=1$, we have
\begin{eqnarray*}
 \mathrm{tr}\, \Tb_{1,1,\mu} &=& \int_{\mathbb{D}} \left[\partial_w\partial_{\overline{w}} \frac{1}{(1-w\overline{w})^2}\right](1-|w|^2)^{2k} dA(w) = \int_{\mathbb{D}} \frac{1+2|w|^2}{(1-|w|^2)^{4-2k}}dA(w) \\
&=& \int_0^1 \frac{1+2r}{(1-r)^{4-2k}}dr = \frac{k}{(k-1)(2k-3)}.
\end{eqnarray*}

}
\end{example}
\begin{example}
{\rm Given a point $z_0 \in \mathbb{D}$, consider the distribution being the derivative of the $\delta$-distribution at $z_0$, i.e., $\partial^{\alpha} \overline{\partial}^{\beta}\delta_{z_0}$. The corresponding sesquilinear form is given by
\begin{equation*}
 \Fb_{\alpha,\beta,z_0}(f,g) = (\partial^{\alpha} \overline{\partial}^{\beta}\delta_{z_0},f\overline{g}) = (-1)^{\alpha+\beta} (\delta_{z_0}, \, \partial^{\alpha}_w f\partial^{\beta}_{\overline{w}}\overline{g})=f^{\a}(z_0)\bar{g}^{\b}(z_0).
\end{equation*}
We introduce now the corresponding Toeplitz operator $\Tb_{\alpha,\beta,z_0}:= \Tb_{\Fb_{\alpha,\beta,z_0}}$, and calculate its Berezin transform
\begin{eqnarray*}
\widetilde{\Tb}_{\alpha,\beta,z_0}(z) &=& \Fb_{\alpha,\beta,z_0}(\k_z,\k_z) =
(-1)^{\alpha+\beta} \left(\delta_{z_0}, \, \partial^{\alpha}_w \frac{1-|z|^2}{(1-\overline{z}w)^2} \partial^{\beta}_{\overline{w}} \frac{1-|z|^2}{(1-z\overline{w})^2}\right) \\
&=& (-1)^{\alpha+\beta} (1-|z|^2)^2 \frac{(\alpha+1)! \overline{z}^{\alpha}}{(1 - \overline{z}z_0)^{2+\alpha}} \frac{(\beta+1)! z^{\beta}}{(1-z\overline{z}_0)^{2+\beta}}.
\end{eqnarray*}
Therefore, by \eqref{trace},
\begin{equation*}
 \mathrm{tr}\, \Tb_{\alpha,\beta,z_0} = \left(\partial^{\alpha} \overline{\partial}^{\beta}\delta_{z_0}, \, \frac{1}{(1-w\overline{w})^2} \right) =
 (-1)^{\alpha+\beta}\left[\partial^{\alpha}_w\partial^{\beta}_{\overline{w}}\frac{1}{(1-w\overline{w})^2} \right]_{w=z_0}.
\end{equation*}
We consider now several particular cases.

Let $z_0=0$, then
\begin{equation*}
\widetilde{\Tb}_{\alpha,\beta,0}(z) = (-1)^{\alpha+\beta} (1-|z|^2)^2 (\alpha+1)! (\beta+1)! \overline{z}^{\alpha}z^{\beta},
\end{equation*}
and, by \eqref{fundamental property},
\begin{equation*}
 \mathrm{tr}\, \Tb_{\alpha,\beta,z_0} = (-1)^{\alpha+\beta} (\alpha+1)! (\beta+1)!
\int_{\mathbb{D}}\overline{z}^{\alpha}z^{\beta} dA(z) =
\begin{cases}
 \alpha!(\alpha+1)!, & \mathrm{if} \ \alpha = \beta \\
 0, & \mathrm{if} \ \alpha \neq \beta
\end{cases}
.
\end{equation*}

Let $\alpha=\beta$, then, by \eqref{fundamental property},
\begin{equation*}
\mathrm{tr}\, \Tb_{\alpha,\alpha,z_0} = [(\alpha+1)!]^2 \int_{\mathbb{D}} \frac{|z|^{2\alpha}}{|1-z\overline{z}_0|^{2(2+\alpha)}}dA(z) = [(\alpha+1)!]^2 \left\| \frac{z^{\alpha}}{(1-z\overline{z}_0)^{2+\alpha}}\right\|^2_{\mathcal{A}^2(\mathbb{D})}.
\end{equation*}
In particular, if $\alpha=\beta=1$,
\begin{equation*}
 \mathrm{tr}\, \Tb_{1,1,z_0} = 4 \left\| \frac{z}{(1-z\overline{z}_0)^3}\right\|^2_{\mathcal{A}^2(\mathbb{D})},
\end{equation*}
while, by \eqref{trace},
\begin{equation*}
 \mathrm{tr}\, \Tb_{1,1,z_0} = \left( \delta_{z_0}, \, \partial_w \partial_{\overline{w}}\frac{1}{(1-w\overline{w})^2} \right) = 2 \frac{1+2|z_0|^2}{(1-|z_0|^2)^4},
\end{equation*}
implying that
\begin{equation*}
 \left\| \frac{z}{(1-z\overline{z}_0)^3}\right\|_{\mathcal{A}^2(\mathbb{D})} = \frac{\sqrt{1+2|z_0|^2}}{\sqrt{2}(1-|z_0|^2)^2}.
\end{equation*}

Let finally $\beta=0$, then
\begin{equation*}
 \widetilde{\Tb}_{\alpha,0,z_0}(z) = (-1)^{\alpha} (1-|z|^2)^2 \frac{(\alpha+1)! \overline{z}^{\alpha}}{(1 - \overline{z}z_0)^{2+\alpha}} \frac{1}{(1-z\overline{z}_0)^2},
\end{equation*}
and, by \eqref{trace},
\begin{equation*}
 \mathrm{tr}\, \Tb_{\alpha,0,z_0} = \left(\partial^{\alpha} \delta_{z_0}, \, \frac{1}{(1-w\overline{w})^2} \right) = (-1)^{\alpha}\frac{(\alpha+1)!\, \overline{z}_0^{\alpha}}{(1-|z_0|^2)^{2+\alpha}}.
\end{equation*}
}
\end{example}

\begin{example}
{\rm
Given $r_0 \in (0,1)$, we consider the distribution $\F = \partial_r \delta_{r_0} \otimes (2\pi)^{-1} d\theta$, $z = re^{i\theta}$. This is the radial derivative of the measure concentrated on the circle with radius $r$. We calculate the trace of the operator $\Tb_{\F}$ using \eqref{trace}:
\begin{equation*}
 \mathrm{tr} \, \Tb_{\F} = \left(\partial_r \delta_{r_0} \otimes (2\pi)^{-1} d\theta, \, \frac{1}{(1-r^2)^2}\right) = - \frac{4r_0}{(1-r_0^2)^3}.
\end{equation*}
}
\end{example}

\end{document}